\documentclass[11pt]{amsart}
\usepackage{
 amsmath, 
 amsxtra, 
 amsthm, 
 amssymb, 
 etex, 
 mathrsfs, 
 mathtools, 
 tikz-cd, 
 bbm,
 xr,
 comment,
 enumitem}
\usepackage[toc]{appendix} 
\usepackage[all]{xy}
\usepackage{hyperref}
\usepackage{url}
\usepackage{tipa}
\usepackage{dsfont}
\usepackage{ textcomp }
\usepackage{stmaryrd} 
\usepackage{amssymb}
\usepackage{mathabx} 
\usepackage{amsmath} 
\usepackage{amscd}
\usepackage{amsbsy}
\usepackage{comment, enumerate}
\usepackage{color}
\usepackage{mathtools,caption}
\usepackage{tikz-cd}
\usepackage{longtable}
\usepackage[utf8]{inputenc}
\usepackage[OT2,T1]{fontenc}
\usepackage{changepage}
\usepackage{commath,enumerate}

\usetikzlibrary{decorations.markings}
\tikzset{->-/.style={decoration={
  markings,
  mark=at position #1 with {\arrow{>}}},postaction={decorate}}}

\usetikzlibrary{calc} 

\DeclareMathOperator{\Gal}{Gal}
\DeclareMathOperator{\rank}{rank}

\DeclarePairedDelimiter\floor{\lfloor}{\rfloor}

\DeclareMathOperator{\Sel}{Sel}
\DeclareMathOperator{\coker}{coker}

\newtheorem{theorem}{Theorem}[section]
\newtheorem*{theorem*}{Theorem}
\newtheorem{lemma}[theorem]{Lemma}

\newtheorem{proposition}[theorem]{Proposition}
\newtheorem{corollary}[theorem]{Corollary}

\newtheorem{defn}[theorem]{Definition}

\numberwithin{equation}{section}
\newtheorem{lthm}{Theorem} 

\theoremstyle{remark}
\newtheorem{remark}[theorem]{Remark}

\newcommand\EatDot[1]{}

\newcommand{\ZZ}{\mathbb{Z}}

\definecolor{Green}{rgb}{0.0, 0.5, 0.0}

\newcommand{\Q}{\mathbb{Q}}

\newcommand{\Z}{\mathbb{Z}}

\newcommand{\rk}{\textup{rank}}
\newcommand{\rko}{\textup{rank}_{\mathcal{O}_K}}

  \DeclareFontFamily{U}{wncy}{}
  \DeclareFontShape{U}{wncy}{m}{n}{<->wncyr10}{}
  \DeclareSymbolFont{mcy}{U}{wncy}{m}{n}
  \DeclareMathSymbol{\sha}{\mathord}{mcy}{"58}
  \DeclareMathSymbol{\zhe}{\mathord}{mcy}{"11}

\renewcommand{\Im}{\textup{im}}
\title[Asymptotic forumla for Tate-Shafarevich groups]{An asymptotic formula for Tate-Shafarevich groups of CM elliptic curves at supersingular primes}
\author{Katharina M\"uller}
\address[Müller]{Institut für Theoretische Informatik, Mathematik und Operations Research, Universität der Bundeswehr München, Werner-Heisenberg-Weg 39, 85577 Neubiberg, Germany}
\email{katharina.mueller@unibw.de}
\begin{document}
\maketitle
\begin{abstract}
    Let $K$ be an imaginary quadratic field and $E/\Q$ an elliptic curves with complex multiplication by $\mathcal{O}_K$. Let $K_\infty/K$ be the anticyclotomic $\Z_p$-extension of $K$ and $K_n$ the intermediate layers. Under additional assumptions on Kobayashi's signed Selmer groups we prove an asymptotic formula for $\sha(E/K_n)$. 
\end{abstract}
\section{Introduction}
Throughout this article $p\ge 5$ is a prime. Let $k$ be a number field and $k_\infty/k$ be a $\Z_p$-extension. Let $k_n$ be the unique subextension of degree $p^n$ and $h_n$ the $p$-class number of $k_n$. Iwasawa proved in his seminal paper \cite{iwasawa73} the asymptotic formula
\[v_p(h_n)=\mu p^n+\lambda n+\nu \quad n\gg 0,\]
for invariants $\mu,\lambda\ge 0$ and $\nu\in \ZZ$. 

In light of this key result the analysis of arithmetic objects along $\Z_p$-extensions has become a central topic in modern Iwasawa theory. Mazur \cite{mazur72} generalized Iwasawa's ideas and applied them to elliptic curves with good ordinary reduction at all primes above $p$. He showed that -- if the $p$-primary Selmer group over $k_\infty$ is cotorsion over the Iwasawa algebra of $\Gal(k_\infty/k)$-- there is an asymptotic formula
\[v_p(\vert \sha(E/k_n)\vert)=\mu p^n+\lambda n+\nu \quad n \gg 0,\]
where $\sha(E/k_n)$ is the Tate-Shafarevich group of $E$ over $k_n$ (assuming that it is finite) . A crucial step in his argument is a so called control theorem. For supersingular primes this control theorem is no longer valid. Let $E/\Q$ be an elliptic curve supersingular at $p$ and let $\Q_\infty$ be the cyclotomic $\Z_p$-extension of $\Q$. As $p\ge 5$, this implies that $a_p=0$ by the Hasse bound \cite[Chapter V, Theorem 1.1]{silverman}. Kobayashi \cite{kobayashi03} constructed plus/minus Selmer groups that satisfy a control theorem. From his control theorem he was able to derive the following asymptotic formula
\begin{align*}
    v_p(\vert \sha(E/\Q_n)\vert)=\sum_{k=0}^{\lfloor\frac{n-2}{2}\rfloor}p^{n-1-2k}-\lfloor \frac{n}{2}\rfloor+\lfloor \frac{n+1}{2}\rfloor \lambda^++\lfloor \frac{n+1}{2}\rfloor \lambda^-\\-nr+\sum_{k=1}^{\lfloor \frac{n}{2}\rfloor}\phi(p^{2k})\mu^++\sum_{k=1}^{\lfloor\frac{n+1}{2}\rfloor}\phi(p^{2k-1})\mu^-+\nu, \quad n\gg 0
\end{align*}
where $r=\rk(E(\Q_\infty))<\infty$ and $\phi$ denotes the Euler $\phi$-function. The invariants $\mu^\pm$ and $\lambda^\pm$ are the Iwasawa invariants of the Pontryagin duals of the  plus/minus Selmer groups. 

Instead of the cyclotomic $\Z_p$-extension we consider the anticyclotomic $\Z_p$-extension for the rest of the paper. We keep the assumption that $E/\Q$ is an elliptic curve and that $p$ is a supersingular prime. Let $K$ be an imaginary quadratic field. Let $K_\infty/K$ be the anticyclotomic $\Z_p$-extension, i.e. the $\Z_p$-extension on which $\Gal(K/\Q)$ acts as $-1$, and let $K_n$ be the unique subextension of degree $p^n$. Assume that $K$ satisfies the generalized Heegner hypothesis: Let $N=N_1N_2$ be the conductor of $E$, where $N_1$ and $N_2$ are coprime and $N_2$ is square-free. Assume that all primes dividing $pN_1$ are split in $K$ and that all primes dividing $N_2$ are inert in $K$.  In particular, we assume that $p$ splits in $K$. In this setting the rank of $E(K_n)$ is unbounded and the plus/minus Selmer groups are no longer cotorsion. Nevertheless there is-- under the assumption that $\sha(E/K_n)$ is finite and that the representation
\[\rho \colon G_\Q\to \textup{Aut}(T_p(E))\]
is surjective-- an asymptotic formula \cite{Leilimmuller}:
\begin{align*}v_p(\vert \sha(E/K_n)\vert)=\sum_{k\le n, k\textup{ even}}\mu^+\phi(p^k)+\sum_{k\le n, k\textup{ odd}}\mu^-\phi(p^k)\\+\lfloor \frac{n}{2}\rfloor \lambda^++\lfloor \frac{n+1}{2}\rfloor \lambda^-+\nu \quad n\gg 0.\end{align*}
These Iwasawa invariants are no longer the ones of the plus/minus Selmer groups.

The Heegner hypothesis excludes the case of CM elliptic curves with complex multiplication by $\mathcal{O}_K$ as primes of good supersingular reduction are inert in $K$. The aim of the present paper is to consider this case. Let $\varepsilon\in \{\pm 1\}$ be the root number of $E/\Q$. Then the $-\varepsilon$-Selmer group  is cotorsion while the $\varepsilon$-Selmer group is not. Burungale, Kobayashi and Ota \cite[Theorem 1.1]{BKO} prove that for all $n$ large enough such that $(-1)^n=-\varepsilon$ one has
\begin{align}
\label{bko}
    v_p(\vert \sha(E/K_n) \vert )-v_p(\vert\sha(E/K_{n-1})\vert)=\lambda+\mu \phi(p^n)
\end{align}
for some invariants $\mu,\lambda\ge 0$.

The invariants occuring in the asymptotic formula of Burungale-Kobayashi-Ota  come from the fine Selmer groups, the $-\varepsilon$ Selmer group, and a finitely generated $\Z_p$-module $A$ independent of $n$.

Let $\Lambda$ be the Iwasawa algebra of $\Gal(K_\infty/K)$ over the ring $\mathcal{O}$, the ring of integers of $K_p$, where $K_p$ denotes the completion of $K$ at $p$. . Define
\[\omega^+_n(T)=\prod_{1\le k\le n, k\textup{ even}}\Phi_k(T+1)\quad \omega_n^-=T\prod_{1\le k\le n,k \textup{ odd}}\Phi_k(T+1),\]
where $\Phi_k$ denotes the $p^k$-th cyclotomic polynomial.
Our main result covers the remaining steps:
\begin{lthm}
\label{Thm:a}
    Assume that $\sha(E/K_n)$ and the fine Selmer group $\Sel^0(E/K_n)[\omega_n^{-\varepsilon}]$ are finite for all $n$. Then for all $n$ large enough and such that $(-1)^n=\varepsilon$ one has 
    \[v_p(\vert \sha(E/K_n) \vert )-v_p(\vert\sha(E/K_{n-1})\vert)=\lambda+\mu \phi(p^n).\]
    The integers $\mu$ and $\lambda$ are the Iwasawa invariants of the fine Tate-Shafarevich groups. 
\end{lthm}
As an immediate corollary we obtain
\begin{lthm}
     Assume that $\sha(E/K_n)$ and $\Sel^0(E/K_n)[\omega_n^{-\varepsilon}]$ are finite for all $n$. Then for all $n$ large enough one has
     \begin{align*}&v_p(\vert\sha(E/K_n)\vert)\\&=\begin{cases}
        \mu^{-\varepsilon}\sum_{m\le n, (-1)^m=-\varepsilon}\phi(p^m)+\mu^{\varepsilon}\sum_{m\le n, (-1)^m=\varepsilon}\phi(p^m)\\{}+\lambda^\varepsilon\floor{\frac{n}{2}}+\lambda^{-\varepsilon}\floor{\frac{n-1}{2}}&(-1)^n=-\varepsilon\\
         \\\mu^{-\varepsilon}\sum_{m\le n, (-1)^m=-\varepsilon}\phi(p^m)+\mu^{\varepsilon}\sum_{m\le n, (-1)^m=\varepsilon}\phi(p^m)\\{}+\lambda^\varepsilon\floor{\frac{n-1}{2}}+\lambda^{-\varepsilon}\floor{\frac{n}{2}}&(-1)^n=\varepsilon
    \end{cases}\end{align*}
\end{lthm}
The invariants are the ones from \eqref{bko} and Theorem \ref{Thm:a} respectively. Note that one expects $\mu^\pm=0$ in this setting. 

\begin{remark} The condition that $ẞSel^0(E/K_n)[\omega_n^{-\varepsilon}]$ is finite is equivalent to the statement that the characteristic ideal of $\Sel^0(E/K_\infty)^\vee$ is corpime to $\omega_n^{-\varepsilon}$.

    If one assumes that $\sha(E/K_n)$ is finite, this is equivalent to
    \[f_n=\frac{\rank(E(K_n))-\rank(E(K_{n-1}))}{2\phi(p^n)}\le 1,\]
    for all $n$ such that $(-1)^n=-\varepsilon$. It is known that $f_n=0$ for all such $n$ large enough \cite{Gre_PCMS}.
\end{remark}
The central idea of the proof is to decompose $\sha(E/K_n)$ into plus and minus Tate-Shafarevich groups whose intersection is the fine Tate-Shafarevich group. Using control theorems for the respective Selmer groups we will then derive the above asymptotic formula. This approach differs from the one presented in \cite{BKO}. In \emph{loc. cit} the authors relate the growth of $\sha(E/K_n)$ to the cokernel of 
\[\Sel(E/K_{n-1})\to \Sel(E/K_n).\]
If $(-1)^n=-\varepsilon$ this cokernel is finite and compuatable in terms of Iwasawa invariants. In the case $(-1)^n=\varepsilon$, this cokernel is of corank $\phi(p^n)$ for all $n$. We have thus to apply different methods and need the additional assumption that $\Sel^0(E/K_n)[\omega_n^{-\varepsilon}]$ is finite for all $n$. 

The fine Tate-Shafarevich groups do not only play a central role in our proofs, but are also of independent interest and we are able to derive an asymptotic formula for them.
\begin{lthm}
\label{thm}
    Let $\kappa^0(E/K_n)$ be the fine Tate-Shafarevich group  of $E$ over $K_n$. For all $n\gg 0$ we have
    \[v_p(\vert \kappa^0(E/K_n)\vert)=\lambda n+p^n\mu +\nu,\]
    for $\mu,\lambda\ge 0$ and $\nu\in \Z$.
\end{lthm}
Note that Theorem \ref{thm} is a generalization of the results in \cite{Lim-control}. \emph{Loc. cit.} only conciders the cases of good ordinary reduction (Theorem 1.7) and of $E(K_\infty)$ being of finite rank (theorem 1.6). Both conditions are not satisfied for supersingular elliptic curves and the anticyclotomic $\Z_p$-extension. 

\section*{Acknowledgments}
The author would like to thank Ben Forrás, Antonio Lei, Meng Fai Lim and Andreas Nickel for helful comments on an earlier draft of this article.

\section{plus/minus Selmer groups}
Let $K$ be an imaginary quadratic field and let $p$ be prime that is inert in $K$. Let $E/\Q$ be an elliptic curve that has complex multiplication by $\mathcal{O}$. Let $K_\infty/K$ be the anticyclotomic $\Z_p$-extension and let $K_n$ be the intermediate fields. Let $\varepsilon$ be the root number of $E$. Let $\tau$ be a topological generator of $\Gal(K_\infty/K)$, let $T=\tau-1$ and $\Lambda=\mathcal{O}\llbracket T\rrbracket$. Throughout the paper we assume that $\sha(E/K_n)$ is finite for all $n$. 

Let $\Xi$ denote the set of Dirichlet characters of $\Gal(K_\infty/K)$. Let $\Xi^+$ be the subset of non-trivial characters whose order is an even power of $p$ and let $\Xi^-$ be the set of characters whose order is an odd power of $p$ and the trivial character. Let $K_{n,p}$ be the localization of $K_n$ at the unique prime above $p$ in $K_n$. For any character $\chi\in \Xi$ and any $x\in \widehat{E}(K_{n,p})$ we define
\[\lambda_\chi(x)=p^{-m}\sum_{\sigma\in \Gal(K_{m,p}/K_p)}\log(x^\sigma)\chi^{-1}(\sigma),\]
where $\chi$ is a character factoring through $\Gal(K_{m,p}/K)$ and $m\ge n$. Let 
\[\widehat{E}(K_{n,p})^\pm=\{x\in \widehat{E}(K_{n,p})\mid \lambda_\chi(x)=0\quad \forall \chi\in \Xi^\mp\}.\]

Let $\Sigma$ the set of primes dividing the conductor of $E$ and $p$. Let $K_\Sigma$ be the maximal Galois extension of $K$ unramified outside $\Sigma$. 
\begin{defn} 
We define
\begin{align*}\Sel(E/K_n)=\ker\left(H^1(K_\Sigma/K_n,E[p^\infty])\to \prod_{v\in \Sigma,(v,p)=1}H^1(K_{n,v},E[p^\infty])\times \frac{H^1(K_{n,p},E[p^\infty])}{\widehat{E}(K_{n,p})\otimes \Q_p/\Z_p}\right)\end{align*}
We define the plus/minus Selmer groups
\[\Sel^\pm(E/K_n)=\ker\left(H^1(K_\Sigma/K_n,E[p^\infty])\to \prod_{v\in \Sigma,(v,p)=1}H^1(K_{n,v},E[p^\infty])\times \frac{H^1(K_{n,p},E[p^\infty])}{\widehat{E}^\pm(K_{n,p})\otimes \Q_p/\Z_p}\right)\]
    as well as the fine Selmer group
    \[\Sel^0(E/K_n)=\ker\left(H^1(K_\Sigma/K_n,E[p^\infty])\to \prod_{v\in \Sigma}H^1(K_{n,v},E[p^\infty])\right).\]
    For $*\in \{0,+,-\}$ we define
    \[\mathcal{M}^*(E/K_n)=\Sel^*(E/K_n)\cap (E(K_n)\otimes \Q_p/\Z_p).\]
    The intersection is taken in $H^1(K_\Sigma/K_n,E[p^\infty])$ and $E(K_n)\otimes \Q_p/\Z_p$ is a subgroup after applying the Kummer map. 
   We furthemore define
    \[\kappa^*(E/K_n)=\frac{\Sel^*(E/K_n)}{\mathcal{M}^*(E/K_n)}.\]
    Let $\Sel^*(E/K_\infty)=\varinjlim_n\Sel^*(E/K_n)$.
\end{defn}
\begin{remark}
    By \cite{BKO24} Lemma 2.2 $H^1(K_{n,v},E[p^\infty])=0$ for all $v$ coprime to $p$. Thus, one can omit the conditions at primes away from $p$ in the definition of Selmer groups.
\end{remark}
In the following we will analyze the $\varepsilon$-Selmer groups.
\begin{lemma}
\label{lem-corank}
    $\left(\frac{\Sel^\varepsilon(E/K_\infty)}{\Sel^0(E/K_\infty)}\right)$ has $\Lambda$-corank one.
\end{lemma}
\begin{proof}
    By \cite[Proposition 3.4]{BKO24} $\Sel^0(E/K_\infty)$ is $\Lambda$-cotorsion. By \cite[Theorem 3.6]{agboolahowardsupersingular} $\Sel^\varepsilon(E/K_\infty)$ has $\Lambda$-corank one. Both results together imply the desired result. 
\end{proof}
\begin{lemma}
    $\left(\frac{\Sel^\varepsilon(E/K_\infty)}{\Sel^0(E/K_\infty)}\right)^\vee\cong \Lambda$. 
\end{lemma}

\begin{proof}
    By \cite[Theorem 3.2 and Lemma 3.3]{BKO24} we have
    \[(E^\pm(K_{n,p})\otimes \Q_p/\Z_p)^\vee\cong \omega^{\mp}_n\Lambda_n,\]
where $\Lambda_n=\mathcal{O}[\Gal(K_n/K)]$.    Taking the projective limit, we obtain 
    \[(E^\pm(K_{\infty,p})\otimes \Q_p/\Z_p)^\vee\cong \Lambda.\]
    By definition 
    \[\left (\frac{\Sel^\varepsilon(E/K_\infty)}{\Sel^0(E/K_\infty)}\right)\hookrightarrow E^\varepsilon(K_\infty)\otimes \Q_p/\Z_p. \]
    This implies that we have a natural surjection 
    \[\Lambda \to \left (\frac{\Sel^\varepsilon(E/K_\infty)}{\Sel^0(E/K_\infty)}\right)^\vee.\]
    As the latter module has $\Lambda$-rank one by Lemma \ref{lem-corank}, this map is actually an isomorphism. 
\end{proof}
As an immediate corollary we obtain
\begin{corollary} The natural map
    \[\Sel^\varepsilon(E/K_\infty)\to E^\varepsilon(K_{\infty,p})\otimes \Q_p/\Z_p\]
    is a surjection. 
\end{corollary}
For all $n\ge 0$ we define
\[e_n=\frac{\rko(E(K_n))-\rko(E(K_{n-1}))}{ \phi(p^n)} \]
\begin{lemma}
    Assume that $(-1)^n=\varepsilon$. Then $e_n\ge 1$.
\end{lemma}
\begin{proof}
    By \cite[Theorem 5.2]{agboolahowardsupersingular} there is an injective homomorphism with finite cokernel
    \[\Sel^\varepsilon(E/K_n)[\omega_n^\varepsilon]\to \Sel^\varepsilon(E/K_\infty)[\omega_n^\varepsilon].\]
    By Lemma \ref{lem-corank} $\Sel^\varepsilon(E/K_\infty)$ has a quotient that is isomorphic to $\Lambda^\vee$. It follows that $\Sel^\varepsilon(E/K_n)^\vee\otimes \Q_p$ contains a submodule iomorphic to $\Lambda/(\omega_n^\varepsilon/\omega_{n-1}^\varepsilon)\otimes \Q_p$. As the Tate-Shafarevich group is assumed to be finite for all $n$, this implies that $E(K_n)\otimes \Q_p$ contains a submodule isomorphic to $\Lambda/(\omega_n^\varepsilon/\omega_{n-1}^\varepsilon)\otimes \Q_p$. As $E(K_{n-1})$ is annihilated by $\omega_{n-1}$, we obtain the desired result.  
\end{proof}
\begin{corollary}
\label{cor.surj}
    The natural homomorphisms
    \[\Sel^\varepsilon (E/K_n)\to \widehat{E}^\varepsilon(K_{n,p})\otimes \Q_p/\Z_p\]
    and 
    \[\mathcal{M}^\varepsilon (E/K_n)\to \widehat{E}^\varepsilon(E)(K_{n,p})\otimes \Q_p/\Z_p\]
    are surjective. In particular,
    \[\frac{\Sel^\varepsilon(E/K_n)}{\Sel^\varepsilon(E/K_{n-1})+\Sel^0(K_n)}\cong \frac{\widehat{E}^\varepsilon(K_{n,p})\otimes \Q_p/\Z_p}{\widehat{E}^\varepsilon(K_{n-1,p})\otimes \Q_p/\Z_p}\cong \frac{\mathcal{M}^\varepsilon(E/K_n)}{\mathcal{M}^\varepsilon(E/K_{n-1})+\mathcal{M}^0(K_n)}\]
\end{corollary}
\begin{proof}
    By \cite[theorem 3.2 (1)]{BKO24} $\widehat{E}^\varepsilon(K_{n,p})\otimes \Q_p\cong \Q_p[X]/\omega_n^\varepsilon$. As $e_n\ge 1$ for all $n$ with $(-1)^n=\varepsilon$, $E(K_n)\otimes \Q_p$ contains a subrepresentation isomorphic to $\Q_p[X]/\omega_n^\varepsilon$. 
    Thus, $\widehat{E}^\varepsilon(K_{n,p})\otimes \Q_p$ lies in the image of the natural homomorphism
    \[E(K_n)\otimes \Q_p\to \widehat{E}(K_{n,p})\otimes \Q_p.\]
    It follows that
    \[\Sel^\varepsilon(E/K_n)\to \widehat{E}^\varepsilon(K_{n,p})\otimes \Q_p/\Z_p\]
    and 
    \[\mathcal{M}^\varepsilon(E/K_n)\to \widehat{E}^\varepsilon(K_{n,p})\otimes \Q_p/\Z_p\]
    have finite cokernels. As the right hand site is divisible, the homomorphism has to be surjective.    
\end{proof}
The remainder of the section is dedicated to prove the existence of a short exact sequence
\[0\to \kappa^0(E/K_n)\to \kappa^+(E/K_n)\oplus \kappa^-(E/K_n)\to \sha(E/K_n)\to 0.\]
\begin{proposition}
\label{prop-selmer-exact}
    We have two short exact sequences
    \[0\to \Sel^0(E/K_n)\to \Sel^+(E/K_n)\oplus \Sel^-(E/K_n)\to \Sel(E/K_n)\to 0\]
    and 
    \[0\to \mathcal{M}^0(E/K_n)\to \mathcal{M}^+(E/K_n)\oplus \mathcal{M}^-(E/K_n)\to E(K_n)\otimes \Q_p/\Z_p\to 0\]
\end{proposition}
\begin{proof}
    For the first claim it suffices to show that \[\Sel^+(E/K_n)+\Sel^-(E/K_n)=\Sel(E/K_n).\] 
 As $\widehat{E}(K_{n,p})\otimes \Q_p/\Z_p=(\widehat{E}^+(K_{n,p})\otimes \Q_p/\Z_p)\oplus (E^-(K_{n,p})\otimes \Q_p/\Z_p)$ (c.f. \cite[Theorem 3.2]{BKO24}), Corollary \ref{cor.surj} implies that indeed \[\Im(\Sel^+(E/K_n)+\Sel^-(E/K_n))=\Im(\Sel(E/K_n)),\]
 where $\Im(\cdot)$ denotes the image inside $E(K_{n,p})\otimes \Q_p/\Z_p$. As $\Sel^0(E/K_n)\subset \Sel^\pm(E/K_n)$, the first claim follows.

 The second claim can be proved similarly. 
\end{proof}
As an immediate Corollary we obtain
\begin{corollary}\label{cor.exact.sha} We have a short exact sequence
    \[0\to \kappa^0(E/K_n)\to \kappa^+(E/K_n)\oplus \kappa^-(E/K_n)\to \sha(E/K_n)\to 0\]
\end{corollary}
\begin{proof}
    Consider the following commutative diagram:
    \[\begin{tikzcd}[font=\small, column sep=1em, row sep=1em]
        0\arrow[r]&\mathcal{M}^0(E/K_n)\arrow[r]\arrow[d]&\mathcal{M}^+(E/K_n)\oplus \mathcal{M}^-(E/K_n)\arrow[r]\arrow[d]& E(K_n)\otimes \Q_p/\Z_p\arrow[r]\arrow[d]&0\\
        0\arrow[r]&\Sel^0(E/K_n)\arrow[r]&\Sel^+(E/K_n)\oplus \Sel^{-}(E/K_n)\arrow[r]&\Sel(E/K_n)\arrow[r]&0
    \end{tikzcd}\]
    The vertical maps are all injective. The claim now follows from the snake lemma.
\end{proof}
\section{Plus/Minus Tate-Shafarevich groups}
In view of Corollary  \ref{cor.exact.sha} it suffices to find assimptotic formulas for $\kappa^+(E/K_n)$ for $*\in \{\pm,0\}$. In teh present section we will concentrate on the comparison of $\kappa^\varepsilon(E/K_n)$ and $\kappa^0(E/K_n)$. 
\begin{defn}
    Let $*\in \{0,+,-\}$ we denote by \[\alpha^*_{n,n-1}\colon \kappa^*(E/K_{n-1})\to \kappa^*(E/K_n)\] the natural map. We define
    \[\kappa^*_{n,n-1}=\frac{\kappa^*(E/K_n)}{\Im(\alpha^*_{n,n-1})}.\]
    Analogosly we define $\sha_{n,n-1}=\frac{\sha(E/K_n)}{\Im (\sha(E/K_{n-1})\to \sha(E/K_n))}$
\end{defn}

\begin{lemma}
\label{lem:Phi_n-inj}
    The natural homomorphism 
    \[\Phi_n\colon\kappa^0_{n,n-1}\to \kappa^\varepsilon_{n,n-1}\]
    is an isomorphism. 
\end{lemma}
\begin{proof}
    By definition,
    \[\coker(\Phi_n)\cong \frac{\Sel^\varepsilon(E/K_n)}{\Sel^\varepsilon(E/K_{n-1})+\mathcal{M}^\varepsilon(E/K_n)+\Sel^0(E/K_n)}=0\]
    by Corollary \ref{cor.surj}. It remains to show that $\Phi_n$ is injective. 

    The kernel of $\Phi_n$ is given by the image of $(\mathcal{M}^\varepsilon(E/K_n)+\Sel^\varepsilon(E/K_{n-1}))\cap \Sel^0(E/K_n)$ in $\kappa^0_{n,n-1}$. Note that 
    \begin{align*}
        &(\mathcal{M}^\varepsilon(E/K_n)+\Sel^\varepsilon(E/K_{n-1}))\cap \Sel^0(E/K_n)\\&\subset (\mathcal{M}^\varepsilon(E/K_{n-1})+\mathcal{M}^0(E/K_n)+\Sel^\varepsilon(E/K_{n-1}))\cap \Sel^0(E/K_n)\\
        &=(\mathcal{M}^0(E/K_n)+\Sel^\varepsilon(E/K_{n-1}))\cap \Sel^0(E/K_n)\\
        &=(\Sel^\varepsilon(E/K_{n-1})\cap \Sel^0(E/K_n))+\mathcal{M}^0(E/K_n)=\Sel^0(E/K_{n-1})+\mathcal{M}^0(E/K_n),
    \end{align*} where the first inclusion follows from the following fact. Let $a\in \mathcal{M}^\varepsilon(E/K_n)$, $b\in \Sel^\varepsilon(E/K_{n-1})$ and $a+b\in \Sel^0(E/K_n)$. Then $\Im (a)\in \widehat{E}^\varepsilon(K_{n-1,p})\otimes \Q_p/\Z_p$, which implies by Corollary \ref{cor.surj} that $a\in \mathcal{M}^0(E/K_n)+\mathcal{M}^\varepsilon(E/K_{n-1})$.

    By definition, the last term in the above equation has trivial image in $\kappa^0_{n,n-1}$. Thus $\Phi_n$ is indeed injective. 
\end{proof}
The next lemma is an preparation to prove the following exact sequence
\[0\to \kappa^0_{n,n-1}\to \kappa^+_{n,n-1}\oplus \kappa^-_{n,n-1}\to \sha_{n,n-1}\to 0.\]
\begin{lemma}
    \label{the-right-sign} Assume that $(-1)^n=\varepsilon$.
    The natural homomorphism
    \[\kappa^\varepsilon_{n,n-1}\to \sha_{n,n-1}\]
    is injective. 
\end{lemma}
\begin{proof}
    Consider the natural homomorphism 
    \[\Psi_n\colon \Sel^{\varepsilon}(E/K_n)\to \sha_{n,n-1}.\]
    The kernel is given by
    \begin{align*}
        &\Sel^{\varepsilon}(E/K_n)\cap (E(K_n)\otimes \Q_p/\Z_p+\Sel(E/K_{n-1}))\\&=(E(K_{n-1})\otimes\Q_p/\Z_p+\mathcal{M}^\varepsilon(E/K_n)+\Sel(E/K_{n-1}))\cap \Sel^\varepsilon(E/K_n)\\&=(\Sel(E/K_{n-1})+\mathcal{M}^\varepsilon(E/K_n))\cap\Sel^\varepsilon(E/K_n)\\&=\Sel^\varepsilon(E/K_n)\cap \Sel(E/K_{n-1})+\mathcal{M}^\varepsilon(E/K_n)\\&=\Sel^\varepsilon(E/K_{n-1})+\mathcal{M}^\varepsilon(E/K_n),
    \end{align*}
    which has a trivial image in $\kappa^\varepsilon_{n,n-1}$.
\end{proof}
\begin{proposition}
\label{prop.exact.sha}
    There is an exact sequence
    \[0\to \kappa^0_{n,n-1}\to \kappa^+_{n,n-1}\oplus \kappa^-_{n,n-1}\to \sha_{n,n-1}\to 0\]
\end{proposition}
\begin{proof}
    Consider the following commutative diagram
    \[\begin{tikzcd}[font=\small, column sep=1em, row sep=1em]
        0\arrow[r]&\kappa^0(E/K_{n-1})\arrow[r]\arrow[d,"a"]&\kappa^+(E/K_{n-1})\oplus \kappa^-(E/K_{n-1})\arrow[r]\arrow[d,"b"]&\sha(E/K_{n-1})\arrow[r]\arrow[d,"c"]&0\\
        0\arrow[r]&\kappa^0(E/K_{n})\arrow[r]&\kappa^+(E/K_{n})\oplus \kappa^-(E/K_{n})\arrow[r]&\sha(E/K_{n})\arrow[r]&0
    \end{tikzcd},\]
    where the rows are exact by Corollary \ref{cor.exact.sha}. Applying the snake lemma we obtain
    \[\kappa^0_{n,n-1}\to \kappa^+_{n,n-1}\oplus \kappa^-_{n,n-1}\to \sha_{n,n-1}\to 0.\]
    The left most homomorphism is injective by Lemma \ref{lem:Phi_n-inj}, which implies the desired short exact sequence. 
\end{proof}
\begin{corollary}
\label{reduction-to-other-sign}
    We have
    \[\vert \sha_{n,n-1}\vert=\vert \kappa^{-\varepsilon}_{n,n-1}\vert.\]
\end{corollary}
\begin{proof}
    This is an immediate consequence of Lemma \ref{lem:Phi_n-inj} and Proposition \ref{prop.exact.sha}.
\end{proof}

\subsection{Estimating $\kappa^0_{n,n-1}$}
Before we continue to estin´mate $\sha_{n,n-1}$ and $\kappa^{-\varepsilon}_{n,n-1}$ we first determine $\kappa^0_{n,n-1}$.
\begin{lemma}
\label{lem:kappa-epsilon}
    The natural homomorphism
    \[\kappa^\varepsilon(E/K_{n-1})\to \kappa^\varepsilon(E/K_n)\]
    is injective.
\end{lemma}
\begin{proof}
    Consider the natural map
    \[\Sel^\varepsilon(K_{n-1})\to \kappa^\varepsilon(K_n).\]
    Its kernel is given by
    \begin{align*}
        &\mathcal{M}^\varepsilon(K_n)\cap \Sel^\varepsilon(K_{n-1})=(\mathcal{M}^\varepsilon(K_{n-1})+\Sel^0(K_{n-1}))\cap \mathcal{M}^\varepsilon(K_n)\\&=\mathcal{M}^\varepsilon(K_{n-1})+\mathcal{M}^0(K_{n-1})=\mathcal{M}^\varepsilon(K_{n-1}),
    \end{align*} where the first equality follows from Corollary \ref{cor.surj}. As the image of $\mathcal{M}^\varepsilon(K_{n-1})$ in $\kappa^\varepsilon(K_{n-1})$ is trivial, the claim follows.
\end{proof}
As an immediate consequence of the above Lemma and Lemma \ref{lem:Phi_n-inj} we obtain
\begin{corollary}
\label{cor:kappa-inj}
    The homomorphism
    \[\kappa^0(K_{n-1})\to \kappa^0(K_n)\]
    is injective.
\end{corollary}
To prove an asymptotic formula for $\kappa^0(E/K_n)$, we need a control theorem for $\Sel^0(E/K_n)$ and $\mathcal{M}(E/K_n)$. 
\begin{theorem}
\label{control-theorem}
    The natural homomorphisms 
    \[\Sel^0(E/K_n)\to \Sel^0(E/K_\infty)^{\Gamma_n}\]
    and
    \[\mathcal{M}^0(E/K_n)\to \mathcal{M}^0(E/K_\infty)^{\Gamma_n}\]
    are injective with uniformly bounded cokernels. 
\end{theorem}
\begin{proof}
    The injectivity follows from the inflation restriction exact sequence and the fact that $E(K_n)[p]=0$. That the first map has uniformly bounded cokernel follows from \cite[Theorem 1.1]{Lim-control}. To show the boundednes of the cokernels for the second homomorphism consider the following commutative diagram
    \[\begin{tikzcd}[font=\small, column sep=1em, row sep=1em]
        0\arrow[r]&\mathcal{M}^0(E/K_n)\arrow[r]\arrow[d,"a_n"]&\Sel^0(E/K_n)\arrow[r]\arrow[d,"b_n"]&\kappa^0(E/K_n)\arrow[d,"c_n"]\arrow[r]&0\\
        0\arrow[r]&\mathcal{M}^0(E/K_\infty)[\Gamma_n]\arrow[r]&\Sel^0(E/K_\infty)[\Gamma_n]\arrow[r]&\kappa^0(E/K_\infty)[\Gamma_n]
    \end{tikzcd}.\]
    By Corollary \ref{cor:kappa-inj} $c_n$ is injective. We therefore obtain an injection $\coker(a_n)\to \coker(b_n)$. As the latter group is uniformly bounded independent of $n$, the same is true for $\coker(a_n)$. 
\end{proof}
Recall that $\Sel^0(E/K_\infty)$ is $\Lambda$-cotorsion. As $\kappa^0(E/K_\infty)$ is a quotient of $\Sel^0(E/K_\infty)$, $\kappa^0(E/K_\infty)$ is $\Lambda$-cotorsion as well. In particular, its Pontryagin dual $\kappa^0(E/K_\infty)^\vee$ is a finitely torsion $\Lambda$-module. 
\begin{theorem}
\label{growth-of-kappa0}
    Let $\mu$ and $\lambda$ be the Iwasawa invariants of $\kappa(K_\infty)^\vee$. Then for all $n$ large enough we have
    \[v_p(\vert \kappa^0_{n,n-1}\vert)=\lambda +\mu \phi(p^n). \]
\end{theorem}
\begin{proof}
   Using standard arguments in Iwasawa theory Theorem \ref{control-theorem} implies that there are invariants
   $\lambda',\lambda'',\mu'$ and $\mu''$ such that 
   \[v_p(\vert\coker(\Sel^0(E/K_{n})\to \Sel^0(E/K_{n-1}))\vert)=\lambda'+\mu'\phi(p^n)\]
   and 
   \[v_p(\vert \coker(\mathcal{M}^0(E/K_{n-1})\to \mathcal{M}^0(E/K_n))\vert )=\lambda''+\mu''\phi(p^n)\]
   for $n\gg 0$.
   By Corollary \ref{cor:kappa-inj} we furthermore have an exact sequence
   \begin{align*}0\to \coker(\mathcal{M}^0(E/K_{n-1})\to \mathcal{M}^0(E/K_n))\to \\\to\coker(\Sel^0(E/K_{n-1})\to \Sel^0(E/K_{n}))\to \kappa^0_{n,n-1}\to 0,\end{align*}
   which implies 
   \[v_p(\vert \kappa^0_{n,n-1}\vert)=\lambda'-\lambda''+(\mu'-\mu'')\phi(p^n).\]
   Let $F'$ and $F''$ be the characteristic ideals of $\Sel^0(E/K_\infty)^\vee$ and $\mathcal{M}^0(E/K_\infty)$. Choose $n_0$ such that $\gcd(F',\omega_n)=\gcd(F', \omega_{n_0})$ and $\gcd(F'',\omega_n)=\gcd(F'',\omega_{n_0})$ for all $n\ge n_0$.  Let $G'=\gcd(F',\omega_{n_0})$ and $G''=\gcd(F'',\omega_{n_0})$ . Then we have $\lambda'=\lambda(F')-\lambda(G')$ as well as $\lambda''=\lambda(F'')-\lambda(G'')$.  
   As we are assuming that $\sha(E/K_n)$ is finite foll all $n$, $\Sel^0(E/K_n)$ and $\mathcal{M}^0(E/K_n)$ have the same corank for all $n$. Thus, $G'=G''$. This implies $\lambda'-\lambda''=\lambda$ and $\mu'-\mu''=\mu$.
   
\end{proof}

\subsection{Estimating the kernels}
To obtain an asymptotic formula for $\sha(E/K_n)$ we do not only need to understand the cokernels $\sha_{n,n-1}$ but also the kernels of the natural maps $\sha(E/K_{n-1})\to \sha(E/K_n)$.  It turns out that these maps are injetive as we will prove in Proposition \ref{shainj}.
\begin{lemma}
\label{mordel-weil-vareps}
    Assume that $(-1)^n=\varepsilon$. For all $n$ large enough we have that 
    \[\mathcal{M}^{-\varepsilon}(E/K_n)=\mathcal{M}^{-\varepsilon}(E/K_{n-1})+\mathcal{M}^0(E/K_n). \]
\end{lemma}
\begin{proof}
    We have a natural isomorphism
    \[\frac{E(K_n)\otimes \Q_p/\Z_p}{E(K_{n-1})\otimes \Q_p/\Z_p}\to \frac{\mathcal{M}^\varepsilon(E/K_n)}{\mathcal{M}^{\varepsilon}(E/K_{n-1})}.\]
    Note that $E(K_{n-1})\otimes \Q_p/\Z_p=\mathcal{M}^\varepsilon(E/K_{n-1})+\mathcal{M}^{-\varepsilon}(E/K_{n-1})$. 

    By definition \begin{align*}
        \mathcal{M}^{-\varepsilon}(E/K_n)&= \mathcal{M}^{-\varepsilon}(E/K_n)\cap (E(K_{n-1})\otimes \Q_p/\Z_p+\mathcal{M}^\varepsilon(E/K_n)))\\&=\mathcal{M}^{-\varepsilon}(E/K_{n})\cap (\mathcal{M}^\varepsilon(E/K_n)+\mathcal{M}^{-\varepsilon}(E/K_{n-1}))\\&=\mathcal{M}^{-\varepsilon}(E/K_{n-1})+(\mathcal{M}^{-\varepsilon}(E/K_n)\cap \mathcal{M}^\varepsilon(E/K_n))\\&=\mathcal{M}^{-\varepsilon}(E/K_{n-1})+\mathcal{M}^0(E/K_n). \qedhere
    \end{align*}
   
\end{proof}

\begin{proposition}
\label{shainj}
    The natural homomorphism 
    \[\sha(E/K_{n-1})\to \sha(E/K_n)\]
    is injective for all $n$ large enough.
\end{proposition}
\begin{proof}
    If $(-1)^n=-\varepsilon$ this follows from  \cite[Lemma 4.5 and Remark 4.6]{BKO24}. Assume now that $(-1)^n=\varepsilon$. By Corollary \ref{cor.exact.sha} it suffices to show that $\kappa^\pm(E/K_{n-1})\to \kappa^\pm (E/K_n)$ is injective. For the $\varepsilon$ part this is Lemma \ref{lem:kappa-epsilon}. It remains to consider the kernel of $\kappa^{-\varepsilon}(E/K_{n-1})\to \kappa^{-\varepsilon}(E/K_n)$. Consider the natural map
    \[\psi_n^{-\varepsilon}\colon \Sel^{-\varepsilon}(E/K_{n-1})\to \kappa^{-\varepsilon}(E/K_n).\]
    The kernel is given by $\Sel^{-\varepsilon}(E/K_{n-1})\cap \mathcal{M}^{-\varepsilon}(E/K_n)$. By Lemma \ref{mordel-weil-vareps} we know
    \[\mathcal{M}^{-\varepsilon}(E/K_n)=\mathcal{M}^{-\varepsilon}(E/K_{n-1})+\mathcal{M}^0(E/K_n).\]
    Thus,
    \begin{align*} \Sel^{-\varepsilon}(E/K_{n-1})\cap \mathcal{M}^{-\varepsilon}(E/K_n)&=\Sel^{-\varepsilon}(E/K_{n-1})\cap (\mathcal{M}^{-\varepsilon}(E/K_{n-1})+\mathcal{M}^0(E/K_n))\\&=\mathcal{M}^{-\varepsilon}(E/K_{n-1})+(\Sel^{-\varepsilon}(E/K_{n-1})\cap\mathcal{M}^0(E/K_n))\\&=\mathcal{M}^{-\varepsilon}(E/K_{n-1})+(\Sel^{0}(E/K_{n-1})\cap \mathcal{M}^0(E/K_n)).\end{align*}
    Corollary \ref{cor:kappa-inj} implies that $\Sel^0(E/K_{n-1})\cap \mathcal{M}^0(E/K_n)=\mathcal{M}^0(E/K_{n-1})$. Thus,
    \[\mathcal{M}^{-\varepsilon}(E/K_{n-1})+(\Sel^{0}(E/K_{n-1})\cap\mathcal{M}^0(E/K_n))=\mathcal{M}^{-\varepsilon}(E/K_{n-1}).\]
    Therefore $\ker(\psi_n^{-\varepsilon})=\mathcal{M}^{-\varepsilon}(E/K_{n-1})$, which implies that 
    \[\kappa^{-\varepsilon}(E/K_{n-1})\to \kappa^{-\varepsilon}(E/K_n)\]
    is indeed injective. 
\end{proof}
\begin{corollary}
\label{cor:sel} For $n\gg 0$ we have
    \[\Sel(E/K_{n-1})\cap (E(K_n)\otimes \Q_p/\Z_p)=E(K_{n-1})\otimes \Q_p/\Z_p. \]
\end{corollary}
\subsection{Estimating $\kappa^{-\varepsilon}_{n,n-1}$}
In this section we always assume that $(-1)^n=\varepsilon$. Before we can analyze $\kappa^{-\varepsilon}_{n,n-1}$, we first need the following result on signed Selmer groups.

\begin{lemma} 
\label{lem:new-control}Assume that the $\Sel^0(E/K_n)[\omega_n^{-\varepsilon}]$ is finite for all $n$. 
    The natural maps
    \[\frac{\Sel^{-\varepsilon}(E/K_n)}{\Sel^0(E/K_n)}[\omega_n^{-\varepsilon}]\to \frac{\Sel^{-\varepsilon}(E/K_\infty)}{\Sel^0(E/K_\infty)}[\omega_n^{-\varepsilon}]\]
    are injective with uniformly bounded cokernel. 
\end{lemma}
\begin{proof}
    Consider the exact sequence
\begin{align*}H^1(K_\Sigma/K_\infty,E[p^\infty])[\omega_n^{-\varepsilon}]\to \left(\frac{H^1(K_\Sigma/K_\infty,E[p^\infty])}{\Sel^0(E/K_\infty)}\right)[\omega_n^{-\varepsilon}]\\\to \Sel^0(E/K_\infty)/\omega_n^{-\varepsilon}\Sel^0(E/K_\infty). \end{align*}
If $\Sel^0(E/K_n)[\omega_n^{-\varepsilon}]$ is finite for all $n$, the characteristic ideal of $\Sel^0(E/K_\infty)$ is coprime to $\omega^{-\varepsilon}_n$ for all $n$. In particular, $\Sel^0(E/K_\infty)/\omega_n^{-\varepsilon}\Sel^0(E/K_n)$ is uniformly bounded. It follows that the natural homomorphism
\[H^1(K_\Sigma/K_\infty,E[p^\infty])[\omega_n^{-\varepsilon}]\to \frac{H^1(K_\Sigma/K_\infty,E[p^\infty])}{\Sel^0(E/K_\infty)}[\omega_n^{-\varepsilon}]\]
is injective with uniformly bounded cokernel.
In particular,
\[\frac{H^1(K_\Sigma/K_n,E[p^\infty])}{\Sel^0(E/K_n)}[\omega_n^{-\varepsilon}]\to \frac{H^1(K_\Sigma/K_\infty, E[p^\infty])}{\Sel^0(E/K_\infty)}[\omega_n^{-\varepsilon}]\]
is injective with uniformly bounded cokernel. Consider the following commutative diagram
\[\begin{tikzcd}[font=\small, column sep=1em, row sep=1em]
    0\arrow[r] &\frac{\Sel^{-\varepsilon} (E/K_n)}{\Sel^0(E/K_n)}\arrow[r]\arrow[d]&\frac{H^1(K_\Sigma/K_n,E[p^\infty])}{\Sel^0(E/K_n)}[\omega_n^{-\varepsilon}]\arrow[r]\arrow[d] &\frac{H^1(K_{n,p},E[p^\infty])}{\widehat{E}^{-\varepsilon}(K_{n,p})\otimes \Q_p/\Z_p}[\omega_n^{-\varepsilon}]\arrow[d]\\
     0\arrow[r] &\frac{\Sel^{-\varepsilon} (E/K_\infty)}{\Sel^0(E/K_\infty)}[\omega_n^{-\varepsilon}]\arrow[r]&\frac{H^1(K_\Sigma/K_\infty,E[p^\infty])}{\Sel^0(E/K_\infty)}[\omega_n^{-\varepsilon}]\arrow[r]&\frac{H^1(K_{\infty,p},E[p^\infty])}{\widehat{E}^{-\varepsilon}(K_{\infty,p})\otimes \Q_p/\Z_p}[\omega_n^{-\varepsilon}]
\end{tikzcd}\]
The right vertical map is injective (this can be proved as in \cite[Theorem 9.2]{kobayashi03} using the Coleman maps defined by \cite{BKO-BDP}). The middle vertical map is injective with uniformly bounded cokernel. Thus, the left vertical map is injective with uniformly bounded cokernel. 
\end{proof}
\begin{lemma}
    Assume that $(-1)^n=\varepsilon$ and that $\Sel^0(E/K_n)[\omega_n^{-\varepsilon}]$ is finite for all $n$. Then We have
    \[\Sel^{-\varepsilon}(E/K_n)=\Sel^{-\varepsilon}(E/K_{n-1})+\Sel^0(E/K_n).\]
\end{lemma}
\begin{proof}
As $\omega_n^{-\varepsilon}=\omega_{n-1}^{-\varepsilon}$ Lemma \ref{lem:new-control} implies that \[\frac{\Sel^{-\varepsilon}(E/K_n)}{\Sel^0(E/K_n)}[\omega_n^{-\varepsilon}]=\frac{\Sel^{-\varepsilon}(E/K_{n-1})}{\Sel^0(E/K_{n-1})}[\omega_{n-1}^{-\varepsilon}].\] As $\frac{\Sel^{-\varepsilon}(E/K_m)}{\Sel^0(E/K_m)}$ is annihilated by $\omega_m^{-\varepsilon}$ for all $m$, we obtain that 
\[\frac{\Sel^{-\varepsilon}(E/K_n)}{\Sel^0(E/K_n)}=\frac{\Sel^{-\varepsilon}(E/K_{n-1})}{\Sel^0(E/K_{n-1})}\]
In particular $\Sel^{-\varepsilon}(E/K_n)=\Sel^{-\varepsilon}(E/K_{n-1})+\Sel^0(E/K_n)$. 
\end{proof}
As an immediate corollary we obtain
\begin{corollary} 
\label{cor.other-sign}Assume that $(-1)^n=\varepsilon$ and that $\Sel^0(E/K_n)[\omega_n^{-\varepsilon}]$ is finite for all $n$. 
    The natural homomorphism 
    \[\kappa^{0}_{n,n-1}\to \kappa^{-\varepsilon}_{n,n-1}\]
    is surjective. 
\end{corollary}
\subsection{Estimating $\sha_{n,n-1}$}
In this section we put the results from previous sections togethzer to obtain an estimate for $\sha_{n,n-1}$ and to derive an asymptotic formula for $\sha(E/K_n)$. 
\begin{theorem}
\label{main-thm}
    Assume that $(-1)^n=\varepsilon$ and that $\Sel^0(E/K_n)[\omega_n^{-\varepsilon}]$ is finite for all $n$. Then we have
    \[\vert \sha_{n,n-1}\vert =\vert\kappa^0_{n,n-1}\vert.\]
\end{theorem}
\begin{proof}
    By Corollarys \ref{reduction-to-other-sign} and \ref{cor.other-sign} we obtain
    \[\vert \sha_{n,n-1}\vert \le \vert \kappa^0_{n,n-1}\vert.\]
    On the other hand Lemmas \ref{lem:Phi_n-inj} and \ref{the-right-sign} imply that there is a chain of injective homomorphisms
    \[\kappa^0_{n,n-1}\to \kappa^\varepsilon_{n,n-1}\to \sha_{n,n-1},\]
    which implies 
    \[\vert \kappa^0_{n,n-1}\vert \le \vert \sha_{n,n-1}\vert.\]
\end{proof}

As a direct consequence of the above analysis we obtain

\begin{theorem} Assume that $\Sel^0(E/K_n)[\omega_n^{-\varepsilon}]$ is finite for all $n$. 
    For all $n$ large enough we have
    \begin{align*}&v_p(\vert\sha(E/K_n)\vert)\\&=\begin{cases}
        \mu^{-\varepsilon}\sum_{m\le n, (-1)^m=-\varepsilon}\phi(p^m)+\mu^{\varepsilon}\sum_{m\le n, (-1)^m=\varepsilon}\phi(p^m)\\{}+\lambda^\varepsilon\floor{\frac{n}{2}}+\lambda^{-\varepsilon}\floor{\frac{n-1}{2}}&(-1)^n=-\varepsilon\\
         \\\mu^{-\varepsilon}\sum_{m\le n, (-1)^m=-\varepsilon}\phi(p^m)+\mu^{\varepsilon}\sum_{m\le n, (-1)^m=\varepsilon}\phi(p^m)\\{}+\lambda^\varepsilon\floor{\frac{n-1}{2}}+\lambda^{-\varepsilon}\floor{\frac{n}{2}}&(-1)^n=\varepsilon
    \end{cases}\end{align*}
\end{theorem}
\begin{proof}
    This is a direct consequence of \cite[Theorem 1.1]{BKO}, Theorem \ref{main-thm}, Proposition \ref{shainj} and Theorem \ref{growth-of-kappa0}.
\end{proof}

\bibliographystyle{amsalpha}
\bibliography{references}
\end{document}